\documentclass[12pt,reqno]{amsart}
\usepackage[dvipsnames]{xcolor}
\usepackage{etoolbox}
\usepackage{mathrsfs}
\usepackage{amssymb,amsmath,amsthm,color}
\usepackage{graphicx, cite}
\usepackage{url}
\usepackage{setspace}
\usepackage{enumerate}

\usepackage{mathtools}
\usepackage{verbatim}
\usepackage{url} 
\usepackage{setspace}
\usepackage{mathtools}
\usepackage[inline]{enumitem}  
\usepackage[margin=1in]{geometry}
\usepackage{comment}
\usepackage{mathabx}
\usepackage{lmodern}
\usepackage{relsize}
\usepackage[T1]{fontenc}

\usepackage{footnote}
\usepackage{cite}
\usepackage{amscd}
\usepackage{color}
\usepackage{euscript}
\usepackage{calc}                   

\definecolor{mypink1}{rgb}{0.858, 0.188, 0.478}
\definecolor{mypink2}{RGB}{219, 48, 122}
\definecolor{mypink3}{cmyk}{0, 0.7808, 0.4429, 0.1412}
\definecolor{mygray}{gray}{0.6}

\usepackage[colorlinks, linkcolor=purple, citecolor=LimeGreen, urlcolor=blue, hypertexnames=false]{hyperref}

\usepackage{tikz}
\usepackage{graphicx} 

\newtheorem{theorem}{Theorem}[section]
\newtheorem{lemma}[theorem]{Lemma}
\newtheorem{proof of lemma}[theorem]{Proof of Lemma}
\newtheorem{proposition}[theorem]{Proposition}

\theoremstyle{definition}

\newtheorem{remark}[theorem]{Remark}

\numberwithin{equation}{section}

\usepackage{color}

\DeclareMathOperator*{\esssup}{ess\,sup}

\begin{document}
\title[Pointwise convergence on Modulation Spaces]
{Pointwise convergence to initial data of heat and Hermite-heat equations in Modulation Spaces}

\author{Divyang G. Bhimani and Rupak K. Dalai}

\address{(Divyang G. Bhimani) Department of Mathematics, Indian Institute of Science Education and Research-Pune, Homi Bhabha Road, Pune 411008, India}

\email{divyang.bhimani@iiserpune.ac.in}

\address{(Rupak K. Dalai) Department of Mathematics, The Assam Royal Global University, Guwahati 781035, India}

\email{rupakinmath@gmail.com}

\subjclass[2020]{Primary 42B25, 33C45, 35C15; Secondary 40A10, 35A01}


\keywords{Pointwise convergence, heat semigroup,  Hermite operator, maximal function, modulation spaces}

\begin{abstract}
We characterize weighted modulation spaces (data space) for which the heat semigroup $e^{-tL}f$   converges  pointwise to the  initial data  $f$ as time $t$ tends to zero. Here $L$ stands for the standard Laplacian $-\Delta $ or Hermite operator $H=-\Delta +|x|^2$
on the  Euclidean space. This is the first result on pointwise convergence with data in a weighted  modulation spaces (which do not coincide with weighted Lebesgue spaces). We also prove that   the Hardy-Littlewood maximal operator operates  on certain modulation spaces. This may be of independent interest.    We have highlighted several open questions that arise naturally from our findings.
\end{abstract}

\maketitle


\section{Introduction}
\subsection{Background and motivation}\label{mb}
Let  $L$ stands for the standard Laplacian $-\Delta$ or Hermite operator (also known as quantum harmonic oscillator) $H=-\Delta +|x|^2$ on $\mathbb R^n.$ We consider the heat equation of the following form
\begin{equation}\label{eq9}
\begin{cases}
    \frac{\partial u}{\partial t}(x,t) = -L\, u(x,t) \\
    u(x,0) = f(x)
\end{cases}
\quad (x,t) \in \mathbb{R}^n \times \mathbb{R}_+.
\end{equation}
The aim of this paper is to study under what condition on initial data $f,$ the solution  $u(x,t)$ to heat equation \eqref{eq9}$-$which is the  heat semigroup as in  \eqref{rsug} and \eqref{H-heatkernel} corresponding to Laplacian and Hermite operator respectively,  converges to $f$  pointwise? Specifically, we wish to characterize the class of data for which  the following limit 
\[\lim_{t\to 0} e^{tL}f(x)= f(x) \quad a.e.\]
holds true. The solution of \eqref{eq9} with $L = -\Delta$ can be written as follows
\begin{eqnarray}\label{rsug}
   u(x,t)= e^{-t\Delta}f(x)= h_t \ast f(x), 
\end{eqnarray}
where $\ast$ denotes convolution on $\mathbb{R}^n$ and the heat kernel 
\begin{equation}\label{heatKernel}
    h_{t}(x) = \frac{1}{(4 \pi t)^{\frac{n}{2}}} e^{-\frac{| x |^{2}}{4t}}.
\end{equation}
The spectral decomposition of the Hermite operator gives 
\begin{equation}\label{sd}
    H = \sum_{k=0}^\infty (2k + n) P_k, \quad P_k f= \sum_{|\alpha|=k} \langle f, \Phi_\alpha \rangle\, \Phi_\alpha,
\end{equation}
where  $\Phi_\alpha $ are the normalized Hermite functions and $P_k$ denotes the orthogonal projection onto the eigenspace corresponding to the eigenvalue $2k + n,$ with the multi-index notation $|\alpha| = \alpha_1 + \dots + \alpha_n$. The Hermite heat  semigroup is given by
\begin{align} 
   e^{-t H} f(x) =& \sum_{k=0}^\infty e^{-t(2k+n)} P_k f(x) = \int_{\mathbb{R}^n} h^H_{t}(x,y) f(y) \, dy, \label{H-heatkernel}
\end{align}
where $h^H_t(x, y)$ denotes the corresponding heat kernel associated with $H$.
See Section \ref{sec5} for details.
The pointwise convergence problem is closely connected with the boundedness of Hardy–Littlewood maximal operator  $\mathcal{M}f$, which is defined by
\begin{equation}\label{eq2}
\mathcal{M}f(x) = \sup_{r > 0} \frac{1}{|B(x, r)|} \int_{B(x, r)} |f(y)|dy,  \quad f\in L^1_{loc}(\mathbb R^n).
\end{equation}
Here \( B(x, r) \) denotes the ball of radius \( r \) centered at \( x \), and \( |B(x, r)| \) its Lebesgue measure. It is  well-known \cite[Chapter V]{MR1232192} that  \(\mathcal{M}\) is bounded on the weighted space \( L^p_w(\mathbb{R}^n)\) ($1<p< \infty$) if and only if \( w \in A_p \), the class of Muckenhoupt weights\footnote{$w\in A_p$ iff $\left(f_{B}\right)^{p} \leq \frac{c}{w(B)} \int_{B} (f)^{p} w d x $ for $f\geq 0$ and ball $B,$ where $f_B=|B|^{-1}\int_Bf(x)dx.$ }.  Here the weighted Lebesgue space norm is given by 
\[\|f\|^p_{L^p_w}=\int_{\mathbb R^n} |f|^p(x)w(x) dx. \]
As a consequence, we have
\begin{equation}\label{ds}
\lim_{t \to 0} e^{t\Delta} f(x) = f(x) \quad  for \ a.e. \  x \in \mathbb{R}^n, \   \forall f \in L^p_w(\mathbb{R}^n).
\end{equation}
Thus,  for such classes of weights, the pointwise convergence holds for the heat semigroup.  See \cite[Theorem 2]{MR0482394}. However, this result gives only a sufficient condition and does not yield a full characterization. On the other hand, the abstract Nikishin theory asserts that almost everywhere pointwise convergence implies the weak-type boundedness of the maximal operator \(\mathcal{M}\) from \(L^p_w(\mathbb{R}^n)\) into \(L^{p, \infty}_u(\mathbb{R}^n)\) for some weight function \(u\); see \cite[Chapter VI]{Jose1985}.
However, establishing strong-type boundedness typically requires a vector-valued framework, as developed in \cite{Rubio81, Mittag, Carleson81}.
In fact, Hartzstein, Torrea, and Viviani \cite{vivianiPAMS} successfully adapted this approach to characterize the weight class 
\begin{equation}\label{Lp-weights}
D^h_p(\mathbb{R}^{n}) := \left\{ v: \mathbb R^{n}\to (0, \infty) \mid \exists\, t_0 > 0 \  \text{ such that  } \  h_{t_0} \in L^{p'}_{v^{-1}}(\mathbb{R}^n) \right\}
\end{equation}
for which \eqref{ds} holds. Specifically, they  established the following result.
\begin{theorem}[\cite{vivianiPAMS}]\label{viviT}
Let $v$ be a weight in $\mathbb{R}^n, 1 \leq p<\infty$, and  heat kernel $\left\{h_t\right\}_t$ be as in \eqref{heatKernel}. Define
$$
\mathcal{W}_R^* f(x)=\sup _{t<R}\left|h_t * f(x)\right|,
$$
for some $R,~ 0<R<\infty$.
The following statements are equivalent:
\begin{enumerate}
\item There exists $0<R<\infty$ and a weight $u$ such that the operator
$f \mapsto \mathcal{W}_R^* f
$
maps $L^p_v$ into $L^p_u$ for $p>1$. In the case $p=1$, it maps $L^1_v$ into weak $L^1_u$.
\item There exists $0<R<\infty$ and a weight $u$ such that the operator
$
f \rightarrow \mathcal{W}_R^* f
$
maps $L^p_v$ into weak $L^p_u$.
\item There exists $0<R<\infty$ such that $h_R * f(x)<\infty$ a.e. $x$ and the limit
$
\lim _{t \rightarrow 0} h_t * f(x)
$
exists a.e. $x$ for all $f \in L^p_v$.
\item There exists $0<R<\infty$ such that
$
\mathcal{W}_R^* f(x)<\infty$,
a.e. $x$, for all $f \in L^p_v$.
\item The weight
$
v \in D_p^h.
$
\end{enumerate}
\end{theorem}
A similar characterization was also obtained for the Poisson semigroup in \cite{vivianiPAMS}.
Since then many authors have studied this a.e. convergence problem in various setting:
\begin{itemize}
\item[-] Heat-diffusion problems associated with the harmonic oscillator and the Ornstein–Uhlenbeck operator, by Abu-Falahah, Stinga, and Torrea \cite{stingaPA};
\item[-] Laguerre-type operators, by Garrigós et al.\ \cite{Laggure2014, Lagguare2017};
\item[-] The Hermite operator \( H = -\Delta + |x|^2 \) and the corresponding Poisson semigroup, by Garrigós et al.\ \cite{torreaTAMS};
\item[-] The Bessel operator, by Cardoso \cite{CardosoJEE};
\item[-] More recently, the problem has been investigated in non-Euclidean settings: Cardoso \cite{CardosoARXIV}, Bruno and Papageorgiou \cite{BrunoARXIV}, and Álvarez-Romero, Barrios, and Betancor \cite{ARBB} studied it on the Heisenberg group, symmetric spaces, and homogeneous trees, respectively. Moreover, Bhimani and Dalai \cite{DR1} have considered the corresponding problem on the torus and waveguide manifolds.
\item[-]  The Laplace operator perturbed by a gradient, the fractional Laplacian, mixed local-nonlocal operators, the Laplacian with inverse square potential,  the Laplacian on Riemannian manifolds, and  the Dunkl Laplacian are  studied by Bhimani-Biswas-Dalai in \cite{BBD}; and have provided general  framework  in the setting  of metric measure spaces. 
\end{itemize}

We conclude this section by highlighting the following remark.
In all the aforementioned results, the characterized class of weights and initial data are  based on  $L^p$ spaces. On the other hand, since the  heat semigroup has been  extensively studied on several other function spaces, it is  natural to study  the a.e. convergence problem there.   The aim of this paper is to study this  problem in the realm of  modulation  spaces.
\subsection{The weight-class in the framework of modulation spaces}
The concept of modulation spaces was introduced in the pioneering work of H. Feichtinger \cite{feichtinger1983modulation} in the early 1980s, and has since become an integral tool in both pure and applied mathematics (see, for example, \cite{wang2011harmonic, HGF2006, KassoBook, FabioBook, GroBook, ToftMI, wang2007global,EFJFA,  FabioJFA, Bhimani2016, CorderoPsDO, IvanTraAMS, cardona2024anharmonic}).
 In order to define these spaces, the idea is to impose integrability conditions on the short-time Fourier transform (STFT) with respect  to a test function from the Schwartz space $\mathcal{S}$. Specifically, 
let $0\neq \phi \in \mathcal{S}$  and 
$f \in \mathcal{S}^{\prime}$(tempered distributions), then 
the STFT of $f$ with respect to $\phi$ is defined as
\begin{equation}\label{eq35}
V_\phi f(x, \xi)=\int_{\mathbb{R}^n} f(y) \overline{\phi(y-x)} 
e^{-2 \pi i x \cdot \xi} d y \quad x,\xi\in\mathbb{R}^n,
\end{equation}
whenever the integral exists. 
Throughout   this paper we assume that weight function $v:\mathbb R^{2n}\to (0, \infty)$ is  \textbf{sub-multiplicative}, that is,  a measurable function  $v : \mathbb{R}^{2n} \to (0,\infty)$ satisfies
$$
 \quad v(z_1+z_2) \leq v(z_1)\, v(z_2) \quad \text{for all} \quad z_1,z_2 \in \mathbb{R}^{2n}.
$$
The \textbf{weighted modulation space} $M_v^{p, q}=M_v^{p, q}(\mathbb{R}^n) \  \ (1\leq p, q \leq \infty)$ is the set of all  $f \in \mathcal{S}^{\prime}
$ for which the following norm
\begin{equation}\label{m-norm}
    \|f\|_{M_v^{p, q}}= \|V_\phi f\|_{L_v^{p, q}}=\left(\int_{\mathbb{R}^n}\left(\int_{\mathbb{R}^n}
\left|V_\phi f(x, \xi) \right|^p v(x,\xi)\, d x\right)^{\frac{q}{p}} d \xi\right)^{\frac{1}{q}}
\end{equation}
 is finite, and with the natural modifications if  $p=\infty$ or $q=\infty.$ 
\begin{remark}\label{urr} The structure of the paper requires several comments. 
\begin{enumerate}
    \item \label{urr1} The definition of the modulation space norm \eqref{m-norm}  is independent of the choice of
the particular window function, see e.g. \cite[Proposition 11.3.2 (c)]{GroBook} and \cite{feichtinger1983modulation}. Thus, regardless of the chosen test function $\phi\in \mathcal{S}$, 
the space $M_v^{p, q}$ remains constant.  In case of $v\equiv1,$ we simply denote $M_1^{p, q}=M^{p, q}.$ See \cite{KassoBook, EFJFA, FabioBook,GroBook, wang2007global, wang2011harmonic} for a comprehensive introduction to these spaces.  
\item \label{urr3} We have not considered the power of weight $v$ in the modulation space norm \eqref{m-norm}, unlike in the existing literature, due to technical reasons. This will be clarified in the below pages (see also Theorem \ref{viviT}).
\item \label{urr2} The concept of a weight function arises frequently in harmonic analysis, which serves to quantify growth, decay, or smoothness. It turns out  that the  correct  class of weight function to  be considered is the ``sub multiplicative weight''  to study the weighted modulation spaces. This contrasts with the weights appearing in the classical setting, e.g. 
the Muckenhoupt classes. The main drawback with  the Muckenhoupt classes is their lack of invariance under time–frequency shifts in weighted modulation spaces.  We refer to excellent survey article  by Gr\"{o}chenig \cite{GroWeight} for a  most relevant classes of weights appearing in the time-frequency analysis.

\item  The   main objective  in this paper is to study pointwise convergence, and so  it is essential that the initial data admit well-defined pointwise values. However, modulation spaces in general consist of tempered distributions, for which pointwise evaluation may not be  meaningful. Moreover, the definition of the maximal function  requires the initial data to be locally integrable.  Thus, throughout the paper we shall assume that the initial data  is also in the space of  locally integrable functions $L^1_{loc}(\mathbb R^n)$. 

\item The STFT \eqref{eq35} is also closely related to Fourier-Wigner and Bargmann transform, see \cite{GroBook}. There is also an equivalent definition of modulation spaces via  uniform frequency decomposition. This is quite similar in  the spirit of Besov spaces, where we have dyadic decomposition, see \cite[Proposition 2.1]{wang2007global} and \cite{feichtinger1983modulation}. See also  \cite{CleanTAMS}  for more unified approach  to these spaces (so called $\alpha$-modulation spaces).  
\end{enumerate}
\end{remark}

Denote $p'$ the H\"older conjugate of $p,$ i.e.  \( 1/p + 1/{p'} = 1 \).
In order to state our main result,  we define following \textbf{weight classes} associated to heat kernel \eqref{heatKernel}:
\[D^h_{p,q}(\mathbb{R}^{2n})= \left\{ v: \mathbb R^{2n}\to \mathbb (0, \infty) \mid\exists\, t_0>0 \  \text{such that  } h_{t_0} \in M^{p', q'}_{v^{-1}} (\mathbb R^n) \right\}.\]

\begin{remark} We compare the weight classes that appears  in this paper.
\begin{enumerate}
    \item  The  class of Muckenhoupt weights $A_p$ is a proper subset of the weight class $D_p^h$ appeared  in Theorem  \ref{viviT}.
    \item  We have imposed  sub-multiplicative assumption for the   member of weight class  $D^h_{p,q}$. As this gives flexibility to use  desired  window functions in   modulation  spaces.  See Remark \ref{urr} \eqref{urr1} and \eqref{urr2}. On the other hand, we do not require this assumption for the class $D^h_{p}$.  
    \item     By fixing a frequency variable $v_{\xi}=v(x, \xi)$,  we can embed $D^{h}_{p,q}(\mathbb R^{2n})$ into $D^{h}_{p}(\mathbb R^n)$. See Lemma \ref{lm6}.
    \item Our formulation of weight class  $D^h_{p,q}$ (defined using modulation spaces) is  inspired from the  weight class $D^h_{p}$ (defined using Lebesgue  spaces, see Theorem \ref{viviT}). This provides the natural way to introduce weighted structures in modulation spaces, and turned out very  effective in our analysis (see Remark \ref{urr}~\eqref{urr3}). 
    \item The Hermite heat kernel \eqref{H-heatkernel} can be estimated by the standard heat kernel; for the upper bound, see \eqref{eq34}, and for the lower bound, see \eqref{eq36}.  Thus the same weight class will work in Theorem \ref{th1} for the both the  Laplacian and Hermite operator. 
\end{enumerate}    
\end{remark}

\section{Main results}
We are now ready to state our main result.
\begin{theorem}\label{th1}
	Let \( v \) be a strictly positive sub multiplicative  weight on $\mathbb{R}^{2n},$ \( 1 \leq p, q < \infty \) and $H=-\Delta + |x|^2.$ Suppose that  $f \in M_v^{p,q}(\mathbb R^n)\cap L^1_{loc}(\mathbb R^n)$ is non-negative.  
    Then  the weight 
    $$ v \in D^h_{p,q}(\mathbb{R}^{2n}) $$
if and only if 
\[\lim_{t \rightarrow 0} h_t \ast f(x) = f(x) \quad for  \ a.e.  \  x\in\mathbb{R}^n \ \]
if and only if 
\[\lim_{t \to 0} e^{-tH} f(x)= f(x)  \quad for  \ a.e.  \  x\in\mathbb{R}^n. \]
\end{theorem}

The boundedness of the local maximal function is not quite clear and seems  challenging in modulation spaces. And so the approach described in Subsection \ref{mb} is not applicable directly to prove Theorem \ref{th1}.

\begin{remark}\label{rsi} The hypothesis of Theorem \ref{th1} deserves several comments.
\begin{enumerate}
    \item For the sufficiency part, the proof relies on maximal function techniques, where the comparison between the semigroup and the Hardy–Littlewood maximal operator requires $f \ge 0$ in order to apply pointwise domination arguments. See Remark \ref{absor} and the proof of Lemmas \ref{lm2} and \ref{lm3}.
    \item For the necessity part, the argument relies on the bounds of the heat kernel (and of the Hermite heat kernel) to derive the required lower estimates for the characterization of the weight class. In this direction, the non-negativity of $f$ is not required.
    \item Due to certain additional algebraic properties of some modulation spaces, specifically when \(1 \leq p < \infty\) and \(1 \leq q \leq 2\), the non-negativity assumption can be relaxed. See Remark~\ref{relaxbso} and Theorem~\ref{th4}.  In general, our method does not allow us to remove the non-negativity hypothesis. Whether the equivalence in Theorem \ref{th1} remains valid for general (possibly sign-changing) functions in $M^{p,q}_v$ is an interesting open question.  See Section \ref{crf}.
\end{enumerate} 
\end{remark}

\begin{remark}[examples] \label{example}  Theorem \ref{th1} is applicable to certain initial data that have not been studied before in the literature \cite{vivianiPAMS,BBD, stingaPA, torreaTAMS}.
\begin{enumerate}
  \item Noticing the following strict embedding
  \[
M^{p, q_1} \subset L^p \subset M^{p, q_2}, \quad q_1 \leq \min \left\{p, p^{\prime}\right\}, \quad q_2 \geq \max \left\{p, p^{\prime}\right\},
\]we deduce the existence of certain non-negative functions \( f \in M^{p, q} \) such that \( f \notin L^p \) for \( p \geq 2 \).
\item It is known (see e.g. \cite[example 2.1]{BhimaniAdv})  that, for $0< \alpha <n,$  we have  \[ f_\alpha(x) = |x|^{-\alpha} \in M^{p,q} \quad  \text{for} \quad p > n / \alpha, \  q > n / (n - \alpha). \]
While $f_{\alpha}$ does not belongs to any Lebesgue spaces $L^p.$
\end{enumerate}
\end{remark}

\begin{remark}  
The Ornstein-Uhlenbeck operator \(\mathcal{O} = -\Delta + 2x \cdot \nabla\) on \(\mathbb{R}^n\) is closely related to a small perturbation of the Hermite operator on $\mathbb{R}^n$, given by  
\[
L = -\Delta + |x|^2 - n.
\]  
Indeed, by defining \(\tilde{u}(x) = e^{-|x|^2 / 2} u(x)\), it follows directly that \(\mathcal{O} u(x) = e^{|x|^2 / 2}(L \tilde{u})(x)\). Consequently, the heat semigroups associated with \(\mathcal{O}\) can be expressed as  
\[
e^{-t \mathcal{O}} f(x) = e^{\frac{|x|^2}{2}} e^{-t L} \tilde{f}(x),
\]  
where \(\tilde{f}(x) = e^{-|x|^2 / 2} f(x)\).  
This connection implies that the convergence properties of the semigroups associated with \(\mathcal{O}\) can be derived from those of \(L\) via the mapping \(f \mapsto \tilde{f}\) (see for e.g. \cite[Eq. (3.1)]{stingaPA} and \cite{torreaTAMS}). As a result, Theorem \ref{th1} for the operator \(\mathcal{O}\) is obtained directly from the corresponding results for \(L\).  
\end{remark}

The study of nonlinear evolution equations with Cauchy data in modulation spaces have gained a lot of interest in recent years. See e.g. \cite{wang2007global, EFJFA, FabioJFA, wang2011harmonic, Bhimani2016}.
We briefly mention well-posedness theory for nonlinear heat equation 
\[u_t +\Delta u = u^{k}\]
with Cauchy data in certain modulation spaces.
Iwabuchi  \cite{iwabuchi2010navier} proved local and  global well-posedness      for small data in some weighted modulation spaces $M^{p,q}_s$. 
Later  Chen-Deng-Ding-Fan \cite{chen2012estimates} have  obtained some space-time estimates for fractional heat  semigroup  $e^{-t (-\Delta)^{\beta/2}}$ in modulation spaces.  Further,  they  \cite[Theorems  1.5 and 1.6]{chen2012estimates}  proved  local and  global well-posedness (for small data)   in some weighted modulation spaces, see also \cite[Theorem 3]{ru2014multilinear}. 
In \cite{huang2016critical, zheng1}  authors have found some critical  exponent in modulation spaces and provide some  local well-posedness and ill-posedness  results. On the other hand, Bhimani \cite{BhiNoDEA} established some finite time blow-up. While Cordero \cite{CorderoPsDO} and  Bhimani et al. in \cite{BhimaniAdv, BhimaniFCAA}  have studied well-posedness theory for the heat equation in $M^{p,q}$ spaces associated with the Hermite operator. Later, Cardona  et al. in \cite{cardona2024anharmonic} have extended  these results for the heat equation associated to anharmonic oscillator. While Trapasso \cite{IvanTraAMS} has studied heat equation associated to the twisted Laplacian in modulation spaces. In light of these considerations and Remark \ref{example}, we would like to mention that Theorem \ref{th1} is new and nicely complements the ongoing research activity.

\medskip

  We shall now briefly comment on the proof of Theorem \ref{th1}. These problems were motivated by the seminal work of Torrea et al. in \cite{torreaTAMS}, where they addressed similar issues in the setting of weighted Lebesgue spaces on $\mathbb{R}^n$. We observe that pointwise convergence results and the boundedness of the maximal operator are closely related. However, to the best of authors’ knowledge, it is important to note that the boundedness of the maximal operator on modulation spaces remains an open problem (Section \ref{moM}), requiring the development of novel techniques, which may be of independent interest. Specifically, to prove Theorem \ref{th1}, we employ various properties of modulation spaces to reduce the problem to a stage where we can apply known results \cite{stingaPA, torreaTAMS, vivianiPAMS} on weighted Lebesgue spaces. We refer to Remark \ref{psrmt} for detail proof strategy.

\subsection{Maximal operator}\label{moM}

The Hardy-Littlewood maximal operator is a fundamental tool in harmonic analysis. The renowned theorem of Hardy, Littlewood, and Wiener asserts that the maximal operator, defined by (\ref{eq2}), is bounded on \( L^p \) for \( 1 < p \leq \infty \), and weak-\( L^1 \) bounded when \( p = 1 \); see \cite{MR0482394, MR663787}. This result is a cornerstone in harmonic analysis due to its profound applications, particularly in potential theory.

\medskip

Our next theorem  shows that  the class of non-negative functions in the modulation space \( M^{p, \infty} \ (1<p\leq \infty) \) is invariant under the action of the Hardy-Littlewood maximal operator. Specifically, we have the following theorem.

\begin{theorem}\label{th7}
Let $1 < p \leq \infty$ and suppose $f \in M^{p,\infty}(\mathbb R^n)\cap L^1_{loc}(\mathbb R^n)$ is non-negative. Then $\mathcal{M}f \in M^{p,\infty}(\mathbb R^n)$.
\end{theorem}

J. Kinnunen in  his seminal paper \cite{MR1469106}  proved that $\mathcal M$ maps 
the Sobolev spaces $W^{1,p} \ (1 < p<\infty)$ into themselves,
using functional analytic techniques. Since $\mathcal M$ 
fails to be bounded in $L^1,$ a vital question was whether the boundedness property holds 
for $f\in W^{1,1}.$ Tanaka \cite{MR1898539} provided a positive answer 
to this in the case of the uncentered maximal function (see definitions therein), which was further improved 
by \cite{MR2276629} whenever $f\in BV$ (set of functions whose total variation is finite). In the centered case (which coincides with \eqref{eq2}), 
Kurka \cite{MR3310075} showed the endpoint question to be true, that is, 
$\mathcal{V}(\mathcal Mf)\leq C\mathcal{V}(f)$ with some constant $C,$
where $\mathcal{V}$ denotes the total variation of a function. However, 
to the best of the author's knowledge, Theorem \ref{th7} is the first result of the maximal operator \eqref{eq2}  established under the modulation space norm.

\begin{remark}
Theorem \ref{th7} encompasses new classes of functions that the boundedness of the maximal operator has not been previously examined. For instance, as discussed in Remark \ref{example}, the function $f_\alpha(x) = |x|^{-\alpha} \in M^{p,q}$ for $0< \alpha <n, \ p > n / \alpha, \  q > n / (n - \alpha),$ which does not belong to any $L^p$ space but lies in $M^{p,q} \subset M^{p,\infty}$, satisfies $\mathcal{M}f \in M^{p,\infty}$.
\end{remark}

\begin{remark}\label{absor} Theorem \ref{th7} remains unknown for arbitrary functions (not necessarily non-negative) within the modulation spaces. This discrepancy arises from the fact that the definition of the maximal operator involves \( |f| \), while certain modulation spaces do not necessarily remain closed under the modulus operation, see  e.g. \cite[Corollary 1.3]{BhimaniNMJ}
and \cite[Corollary 3.2]{Bhimani2016}.
\end{remark}

\begin{remark}
To establish Theorem \ref{th7}, we first show that for any \( f \in M^{p,\infty} \), there exists a Schwartz class function \( \phi \) such that \( \mathcal{M}f * |\phi| \in L^{p} \) (see Lemma \ref{lm2}). Utilizing this result, we can derive the inequality 
$\left|\mathcal{M}f * M_{\xi} \phi(x)\right| \leq \mathcal{M}(f * |\phi|)(x),$
for almost every \( x \), as shown in Lemma \ref{lm3}. These results facilitate the proof of the required boundedness. Since we are employing the \( L^p \) boundedness of the maximal operator result, it is required to consider all non-negative functions within this framework. 
\end{remark}

\begin{remark}\label{relaxbso}
It is worth noting that, by exploiting certain algebraic properties inherent to modulation spaces, Theorem \ref{th7} remains valid for arbitrary functions (not necessarily non-negative) belonging to specific modulation spaces, when $f \in M^{p,q}$ with $1 \leq p < \infty$ and $1 \leq q \leq 2$. Consequently, the pointwise convergence results stated in Theorem~\ref{th1} also hold for such modulation spaces (see Theorem \ref{th4}).
\end{remark}

The paper is organized as follows. In Section \ref{sec2}, we revisit the boundedness of the Hardy-Littlewood maximal operator and examine related results within the framework of Lebesgue spaces. Section \ref{sec3} presents the proof of Theorem \ref{th7}, along with the identification of additional modulation spaces where pointwise convergence is valid for general functions. In Section \ref{sec4}, we determine the relevant weight classes for modulation spaces in the Laplacian setting. In Section \ref{sec5}, we carry out the analysis for the Hermite setting and establish our main result, Theorem \ref{th1}.

\section{Preliminaries}\label{sec2}
A measurable function \( f \) is said to be in the  weak \( L^p \) space if there exists a constant \( C > 0 \) such that for all \( \lambda > 0 \)
	\[
	\left|\{x \in \mathbb{R}^n : |f(x)| > \lambda\}\right| \leq \frac{C^p}{\lambda^p},
	\]
where $|A|$ denotes the Lebesgue measure of the set $A$.
\subsection{Maximal operator on Lebesgue spaces}
\begin{theorem}\em\cite[Theorem 1]{MR1232192} (boundedness of maximal operator)\label{th3} \
\begin{enumerate}
\item If $f \in L^p \ (1 \leq p \leq \infty)$, then the function 
$\mathcal{M}f$ is finite almost everywhere.
\item \label{th3-1} If $f \in L^1$, then 
$ \lambda \cdot\left|\{x:\mathcal Mf(x)>\lambda\}\right| \lesssim_{n} \|f\|_{L^1}$ for every $\lambda>0.$
\item If $f \in L^p \ (1<p \leq \infty)$, 
then $||\mathcal Mf||_{L^p}\lesssim_{p,n}||f||_{L^p}.$
\end{enumerate}
\end{theorem}
\begin{theorem}\em\cite[Theorem 2]{MR0482394}\label{th2} Let $\varphi$ be an integrable function on $\mathbb{R}^n$, and set $\varphi_{t}(x)=t^{-n} \varphi(x / t)$ for $ t>0.$
Suppose the least decreasing radial majorant of $\varphi$ is integrable, i.e. $\psi(x)$ $=\sup_{|y|\geq|x|} |\varphi(y)|$ and $\int_{\mathbb{R}^n} \psi(x) d x=A<\infty$. 
Then
\begin{enumerate}
\item\label{boundedbyM} $\sup_{t>0} \left|\varphi_{t} * f(x)\right| \leq A\, \mathcal Mf(x) \quad\mbox{a.e. for all }
f \in L^p \ (1 \leq p \leq \infty)$.
\item\label{Lp-pw} If  $\int_{\mathbb{R}^n} \varphi(x) d x=1$, 
then $\lim _{t \rightarrow 0}\left(\varphi_{t} * f\right)(x)=f(x)$ a.e. for all $f\in L^p \ (1 \leq p \leq \infty).$
\item If $1\leq p<\infty$, then $\left\|\varphi_{t} * f-f\right\|_{L^p} \rightarrow 0$ as $t \rightarrow 0$.
\end{enumerate}
\end{theorem}
\begin{remark}\cite[ch. V, \S 2.1, pp. 198]{MR1232192}\label{fmf} The following formulations provide equivalent representations of the Hardy–Littlewood maximal function.
\begin{enumerate}
    \item \label{fmf1}  Let  \(\varphi(x) = \frac{1}{|B_1|} \chi_{B_1}(x),\,B_{1} = \{x : |x| < 1\}\). Then the Hardy-Littlewood maximal operator \eqref{eq2} can be represented as 
    \begin{equation}\label{eq29}
        \sup_{t>0} (\varphi_{t} *|f|)(x) =\, \mathcal Mf(x).
    \end{equation}
    \item  Denote 
\[\mathcal{X}=\left\{ \vartheta>0 : \vartheta \text{ is radial and decreasing such that }  \int_{\mathbb{R}^n} \vartheta(x) \, dx = 1 \right\}. \]
 If \(\vartheta \in \mathcal{X}\), then \(\vartheta_{t}(\cdot)= t^{-n} \vartheta(\cdot/t) \in \mathcal{X}\).  
 We may rewrite 
\begin{equation}\label{eq12}
\mathcal{M}f(x) = \sup_{\vartheta \in \mathcal{X}} (\vartheta * |f|)(x).
\end{equation}

\end{enumerate}
\end{remark}

This shows that the maximal function defined in \eqref{eq12} coincides with the classical Hardy- Littlewood maximal function \eqref{eq29}. Indeed, if we denote by $B_t = \{x \in \mathbb{R}^n : |x| < t\}$, then the normalized characteristic function $|B_t|^{-1} \chi_{B_t}$ belongs to the class $\mathcal{X}$, and taking the supremum over such elements recovers the standard definition of $\mathcal{M}f$. Conversely, any element of $\mathcal{X}$ can be approximated by limits of such normalized characteristic functions, thus establishing the equivalence of the two definitions.

\begin{remark}  Theorem \ref{th2} holds for heat   and Poisson kernels corresponding to the standard Laplacian, i.e  we may take  $\varphi_t=h_t$ or $\varphi_t=p_t$ (among others).
\end{remark}

\subsection{Short-time Fourier transform (STFT)}
In the next lemma, we recall several useful facets of the STFT. To this end, for $x,y, \xi \in \mathbb R^n, \epsilon>0,$ denote
\begin{itemize}
    \item $f^*(x)=\overline{f(-x)}$ \>  (involution) 
       \item  $T_yf(x)=f(x-y)$ \>  (translation/time shift by $y$)
      \item   $D_\epsilon f(x)=\epsilon^{-n}f(\epsilon^{-1}x)$ \> ($L^1$-normalized $\epsilon$-dilation) 
       \item  $M_\xi f(x)=e^{2\pi i \xi\cdot x} f(x)$ \> (modulation/frequency shift by $\xi).$ 
\end{itemize}

\begin{lemma}[basic properties]\label{lm1} Let $p_i, q_i \in [1, \infty]$ for $i=0,1,2.$
\begin{enumerate}
\item{\em\cite[Lemma 3.1.1]{GroBook}} If $f, \phi \in L^2,$ then $V_g f$ is uniformly continuous on $\mathbb{R}^{2 n}$, and   
\begin{equation*}\label{eq41}
\begin{aligned}
V_\phi f(x, \xi) & =\left(f \cdot T_x \bar{\phi}\right)^{\wedge}(\xi)  =\left\langle f, M_\xi T_x \phi\right\rangle  =\left\langle\hat{f}, T_\xi M_{-x} \hat{\phi}\right\rangle  =e^{-2 \pi i x \cdot \xi}\left(\hat{f} \cdot T_\xi \overline{\hat{\phi}}\right)^{\wedge}(-x) \\
& =e^{-2 \pi i x \cdot \xi} V_{\hat{\phi}} \hat{f}(\xi,-x)  =e^{-2 \pi i x \cdot \xi}\left(f * M_\xi \phi^*\right)(x)  =\left(\hat{f} * M_{-x} \hat{\phi}^*\right)(\xi),
\end{aligned}
\end{equation*}
where $\hat{\cdot}$ denotes the  Fourier transform.
\item{\em\cite[Theorem 3.2.1]{GroBook} (Moyal identity/orthogonality for STFT)} If $f_i, \phi_i \in L^2 $ for $(i=1,2),$ then  $V_{\phi_i}f_i\in L^2(\mathbb R^{2n})$ and
\begin{equation*}\label{eq37}
\langle V_{\phi_1}f_1, V_{\phi_2}f_2 \rangle = \langle f_1, f_2 \rangle\,  \overline{\langle \phi_1, \phi_2 \rangle}.
\end{equation*}
\item \label{ap} {\em\cite[Corollary 4.21]{KassoBook} } If $\frac{1}{p_{1}}+\frac{1}{p_{2}}=\frac{1}{p_{0}}$ and $\frac{1}{q_{1}}+\frac{1}{q_{2}}=1+\frac{1}{q_{0}},$ then $$\|f g\|_ {M^{p_{0},q_{0}}} \lesssim  \|f\|_ {M^{p_{1},q_{1}}} \|g\|_ {M^{p_{2},q_{2}}}.$$
\end{enumerate}
\end{lemma}

\section{ Hardy-Littlewood maximal operator on modulation spaces}\label{sec3}
In this section, we examine the Hardy–Littlewood maximal operator on modulation spaces. Before proceeding to the proof of Theorem \ref{th7}, we first establish several preliminary results that play a fundamental role in the subsequent analysis. These auxiliary results provide the necessary groundwork and will be addressed in the following section.

\begin{lemma}\label{lm2}
Let $1 < p< \infty,1 \leq q < \infty$ and $f \in M^{p, q}\cap L^1_{loc}$ be 
non-negative. Then there exists some $\phi \in \mathcal S$ 
such that $\mathcal Mf * |\phi|\in L^{p}.$
In the case $p=1,$ $\mathcal Mf * |\phi|$ is in weak-$L^{1}.$
\end{lemma}

\begin{proof} 
In the definition of modulation spaces, since the choice of \(\phi \in \mathcal{S} \setminus \{0\}\) is arbitrary, we may select \(\phi\) such that \(\phi(y) \leq 0\) for a.e. \(y \in \mathbb{R}^n\), and ensure that \(f * |\phi| \in L^p\) by utilizing the following chain of embeddings (see \cite{ToftMI})
\begin{equation}\label{eq42}
M^{p,q} * \mathcal{S} \hookrightarrow M^{p,q} * M^{1, q'} \hookrightarrow M^{p,1} \hookrightarrow L^p.
\end{equation}
By \eqref{eq12}, we have 
\begin{equation}\label{eq14}
\mathcal Mf * \phi(x)=\left(\sup _{\vartheta \in \mathcal{X}} (\vartheta * f)\right)\ast\phi(x)=\int_{\mathbb R^n} \sup _{\vartheta \in \mathcal{X}}(\vartheta * f)(y) \,\phi(x-y) d y.
\end{equation}
It is clear that $\vartheta * f$ is a non-negative function since convolution of two non-negative 
function is non-negative. Hence $(\vartheta * f) (y) \phi(x-y)\leq 0$ for a.e. $y\in\mathbb R^n.$ Recall $\sup A =-\inf (-A)$ for $A\subset \mathbb R.$
Now, we can write
$$\sup _{\vartheta \in \mathcal{X}}(\vartheta * f)(y) \phi(x-y)=
-\inf _{\vartheta \in \mathcal{X}}(\vartheta * f)(y) |\phi|(x-y).$$
By putting the above equation in (\ref{eq14}) and then using \eqref{eq12}, we get

$$\begin{aligned}
\mathcal Mf * |\phi|(x)&=-\left(\mathcal Mf * \phi\right)(x)\\
&=\int_{\mathbb R^n}-\sup _{\vartheta \in \mathcal{X}}(\vartheta * f)(y)\, \phi(x-y)dy\\
&= \int_{\mathbb R^n}\inf_{\vartheta \in \mathcal{X}} \left(\vartheta * f) (y) \,|\phi|(x-y)\right) d y\\
&\leq \int_{\mathbb R^n}\liminf_{n \to \infty} \left(\vartheta_n * f) (y) \,|\phi|(x-y)\right) d y,
\end{aligned}$$
where $\{\vartheta_n\} \subset \mathcal{X}$ be some sequence. Then, by Fatou’s lemma, we obtain
$$\begin{aligned}
\mathcal Mf * |\phi|(x)&\leq\liminf_{n \to \infty}\int_{\mathbb R^n}(\vartheta_n * f) (y) |\phi|(x-y) d y \\
&\leq\sup_{\vartheta \in \mathcal{X}}\int_{\mathbb R^n}(\vartheta * f) (y) |\phi|(x-y) d y \\
& =\sup _{\vartheta \in \mathcal{X}} \left(\vartheta *(f * |\phi|)\right)(x) \\
& =\mathcal M(f * |\phi|)(x).
\end{aligned}$$
Using Theorem \ref{th3}, the $L^p$-boundedness of maximal operator, 
$\mathcal M(f * |\phi|)$ will be in $L^{p}$ for $1<p\leq\infty,$ since 
$f * |\phi|$ is in $L^{p}.$ 

For the case $p=1,$ we can also conclude our claim because 
$\mathcal M(f * |\phi|)$ is weak-$L^1$ whenever $f * |\phi|$ is in $L^1.$
\end{proof}

\begin{lemma}\label{lm3}
Let \( f \in M^{p, q}\cap L^1_{loc},\, 1 \leq p, q < \infty, \) be non-negative. Then, for almost every \( \xi \in \mathbb{R}^n \), the following inequality holds for some $\phi\in\mathcal{S}$
\[
\left|\mathcal{M}f * M_{\xi} \phi(x)\right| \leq \mathcal{M}(f * |\phi|)(x)
\quad \text{for almost every } x \in \mathbb{R}^n.
\]
\end{lemma}

\begin{proof}
By \eqref{eq12}, we have
\begin{align}\label{eq19}
|\mathcal Mf * M_{\xi} \phi(x)|
&\leq\int_{\mathbb R^n} \left|\sup _{\vartheta \in \mathcal{X}}(\vartheta * f)(y)\right|\left|M_{\xi} \phi(x-y) \right|d y\nonumber\\
&=\int_{\mathbb R^n}\sup _{\vartheta \in \mathcal{X}}(\vartheta * f)(y) |\phi|(x-y)dy.
\end{align}
It is clear that for each $\vartheta \in \mathcal{X},$ we have  
\begin{equation}\label{eq20}
|(\vartheta * f)(y) |\phi|(x-y)|=(\vartheta * f)(y) |\phi|(x-y)\leq \mathcal Mf(y) |\phi|(x-y).
\end{equation}
The $L^1$ norm of right hand side of (\ref{eq20}) is
$$\int_{\mathbb R^n}\mathcal Mf(y) |\phi|(x-y)dy=\mathcal Mf*|\phi|(x)<\infty$$
for almost all $x\in\mathbb R^n,$ by Lemma \ref{lm2}.
Hence using dominated convergence theorem, (\ref{eq19}) can be rewritten as
$$
\begin{aligned}
|\mathcal Mf * M_{\xi} \phi(x)|&\leq\int_{\mathbb R^n}\sup _{\vartheta \in \mathcal{X}}(\vartheta * f)(y) |\phi|(x-y)dy\\
&=\sup _{\vartheta \in \mathcal{X}}\int_{\mathbb R^n}(\vartheta * f)(y) |\phi|(x-y)dy\\
& =\sup _{\vartheta \in \mathcal{X}} \left(\vartheta *(f * |\phi|)\right)(x) \\
& =\mathcal M(f * |\phi|)(x)
\end{aligned}$$
for almost every $x\in \mathbb R^n.$
Hence, we conclude the result.
\end{proof}

Now, we are prepared to prove the boundedness result within the modulation space.

\begin{proof}[\textbf{Proof of Theorem \ref{th7}}]
Given that $f \in M^{p, \infty}$ for $1 < p \leq \infty$, showing that $\mathcal Mf \in M^{p, \infty}$ is equivalent to showing (see Lemma \ref{lm1} \eqref{eq41})
\[\esssup_{\xi\in \mathbb{R}^n}\left( \int_{\mathbb{R}^n} \left|\mathcal Mf * M_{\xi} {\phi^*}(x)\right|^p \, dx \right)^{\frac{1}{p}} < \infty.\]
We can choose the \( {\phi} \) in such a way that \( {\phi^*} \) satisfies the required conditions of Lemma \ref{lm3}. By applying Lemma \ref{lm3}, it follows that
$$\left|\mathcal Mf * M_{\xi} {\phi^*}(x)\right| \leq \mathcal M(f * |{\phi^*}|)(x).$$

Furthermore, as per embeddings \eqref{eq42}, we have $f * |{\phi^*}| \in L^{p}.$ So we can derive the following by combining these facts with the $L^p$-boundedness of the maximal operator.

$$\begin{aligned}
& \esssup_{\xi\in \mathbb{R}^n}\left( \int_{\mathbb{R}^n} \left|\mathcal Mf * M_{\xi} {\phi^*}(x)\right|^p \, dx \right)^{\frac{1}{p}}\\
\leq&\esssup_{\xi\in \mathbb{R}^n}\left( \int_{\mathbb{R}^n} \left|\mathcal M
\left(f * |{\phi^*}|\right)(x)\right|^{p} \, dx \right)^{\frac{1}{p}}\\
\lesssim&\esssup_{\xi\in \mathbb{R}^n}\left( \int_{\mathbb{R}^n} \left|f * |{\phi^*}|(x)\right|^{p} \, dx \right)^{\frac{1}{p}}<\infty.
\end{aligned}$$

This establishes that $\mathcal Mf \in M^{p, \infty}$ whenever 
$f \in M^{p, \infty}$ for $1<p \leq \infty.$
\end{proof}

\begin{theorem}\label{th5}
Suppose that the least decreasing radial majorant of $\varphi$ is integrable,
i.e. $\psi(x)=\sup_{|y|\geq|x|} |\varphi(y)|$ and $\int_{\mathbb{R}^n} \psi(x) d x=A<\infty$. Let $1 \leq p, q < \infty$ 
and $f \in M_v^{p, q}\cap L^1_{loc}$ be non-negative. Then
$$\sup_{t>0} \left|\varphi_{t} * f(x)\right| \leq A\, \mathcal Mf(x) ~\mbox{ almost everywhere}.$$
\end{theorem}

\begin{proof}
The proof of this is similar to the $L^{p}$ case as in Theorem \ref{th2}~\eqref{boundedbyM}. When $f \in M_v^{p, q},$  
since $\left|\varphi_{t} * f(x)\right|\leq\psi_{t} * f(x),$ it is sufficient to show
\begin{equation}\label{eq7}
\psi_{t} * f(x) \leq A\,\mathcal{M}f(x)
\end{equation}
holds for every $t>0.$  Hence showing (\ref{eq7}) is the same as showing 
\begin{equation}\label{eq8}
\psi_{t} * (T_xf)(0) \leq A\,\mathcal{M}(T_xf)(0),
\end{equation} where $T_xf(y)=f(y-x).$ 
In order to establish (\ref{eq8}), consider
$$
\begin{aligned}
\psi_{t} * (T_xf)(0)&=\int_{\mathbb{R}^{n}} (T_xf)(y) \psi_t(y) d y\\
& =\int_{0}^{\infty} \int_{S^{n-1}} (T_xf)(r \theta) \psi_t(r) r^{n-1} d \sigma dr \\
& =\int_{0}^{\infty} \lambda(r) \psi_t(r) r^{n-1} d r,
\end{aligned}
$$
where $\lambda(r)=\int_{S^{n-1}} (T_xf)(r \theta) d \sigma$ and $\sigma$ is a surface measure on $S^{n-1}=\{x\in\mathbb{R}^n:|x|=1\}.$ Then, using limiting case and applying integration by parts, we can write
\begin{equation}\label{eq18}
\begin{aligned}
\psi_{t} * (T_xf)(0)&=\lim _{\substack{\varepsilon \rightarrow 0 \\ N \rightarrow \infty}} 
\int_{\varepsilon}^{N} \lambda(r) \psi_t(r) r^{n-1} d r \\
& =\lim _{\substack{\varepsilon \rightarrow 0 \\ N \rightarrow \infty}}
\left(\left. \Lambda(r) \psi_t(r) \right|_{\varepsilon} ^{N}-\int_{\varepsilon}^{N}
\Lambda(r) d\psi_t(r)\right),
\end{aligned}
\end{equation}
where $\Lambda(r)=\int_{0}^{r} \lambda(t) t^{n-1} d t.$ Notice that
\begin{equation}\label{eq17}
0\leq cr^n\psi(r)= \psi(r)\int_{\frac{r}{2}<|x|<r}dx\leq\int_{\frac{r}{2}<|x|<r}\psi(x)dx,
\end{equation}
where $c=\Omega(1-\frac{1}{2^n})$ and $\Omega$ 
is the volume of the unit ball. From the fact that $\psi_t$ is in $ L^{1}$ and decreasing, the right-hand side of 
(\ref{eq17}) vanishes as $r \rightarrow 0$ or $r \rightarrow \infty,$ so is $r^{n} \psi_t(r) \rightarrow 0.$
Hence using this observation and the following inequality
$$\Lambda(r)=\int_{|y|<r} (T_xf)(y) d y \leq \Omega r^{n} \mathcal{M}(T_xf)(0),$$
we can show that the error term of (\ref{eq18}), $\Lambda(N)\psi_t(N)-\Lambda(\varepsilon)\psi_t(\varepsilon)$ 
tends to zero as $\varepsilon \rightarrow 0$ and $N \rightarrow \infty.$ Thus, we have
$$\psi_{t} * (T_xf)(0)=\int_{0}^{N} \Lambda(r) d(-\psi_t(r)) \leq \Omega \mathcal{M}(T_xf)(0) \int_{0}^{\infty} r^{n} d(-\psi_t(r)).$$
\end{proof}

Next, we will establish the pointwise convergence result for specific modulation spaces.

\begin{proposition}\label{prop2}
Suppose $\varphi$ is as defined in Theorem \ref{th5}.
Let $1 \leq p, q < \infty$ and $f \in M^{p, q}\cap L^1_{loc}$ 
be  non-negative. Then 
\begin{equation}\label{eq16}
\lim _{t \rightarrow 0}(\varphi_{t} * f)(x)=f(x) \mbox{ almost everywhere}.
\end{equation}
\end{proposition}

\begin{proof}
We begin by considering the following expression
\begin{equation}\label{eq3}
\begin{aligned}
\lim _{t \rightarrow 0} \left(\varphi_{t} * (f * M_{\xi} {\phi^*})\right)(x)
&=\lim _{t\rightarrow 0}\left((\varphi_{t} * f) * M_{\xi} {\phi^*}\right)(x)\\
&=\lim _{t\rightarrow 0}\int_{\mathbb R^n} (\varphi_{t} * f)(y) M_{\xi} {\phi^*}(x-y) d y.
\end{aligned}
\end{equation}

We aim to use the dominated convergence theorem (DCT) on the right-hand side of the equation (\ref{eq3}). To do this, we define a sequence of functions as follows
$$F_t(y):=(\varphi_{t} * f)(y) M_{\xi} {\phi^*}(x-y).$$
Now, applying Theorem \ref{th5}, we obtain the inequality
\begin{equation}\label{eq4}
|F_t(y)|\leq A\,|\mathcal Mf(y) M_{\xi} {\phi^*}(x-y)|\leq A\,\mathcal Mf(y)|{\phi^*}|(x-y),
\end{equation}
for all $t>0.$ In order to utilize DCT, we need the right-hand side of the equation (\ref{eq4}) to be an integrable function. Employing Lemma \ref{lm2}, it can be readily observed that
$$\int_{\mathbb R^n} \mathcal Mf(y)|{\phi^*}|(x-y)dy=\mathcal Mf * |{\phi^*}|(x)<\infty$$
for almost every $x.$ Hence, by DCT, equation \eqref{eq3} yields
\begin{align}\label{eq5}
\lim _{t \rightarrow 0} \left(\varphi_{t} * (f * M_{\xi} {\phi^*})\right)(x)
&=\int_{\mathbb R^n}\lim _{t\rightarrow 0} \left(\varphi_{t} * f\right)(y) M_{\xi} {\phi^*}(x-y) d y\nonumber\\
&=\left(\lim _{t \rightarrow 0} (\varphi_{t} * f)\right) * M_{\xi} {\phi^*}(x).
\end{align}
On the other hand, Theorem \ref{th2}~\eqref{Lp-pw} implies
\begin{equation}\label{eq6}
\lim _{t \rightarrow 0} \left(\varphi_{t} *\left(f * M_{\xi} {\phi^*}\right)\right)(x)=(f * M_{\xi} {\phi^*})(x),
\end{equation}
since $f * M_{\xi} {\phi^*} \in L^{p}.$
Now, comparing equations (\ref{eq5}) and (\ref{eq6}), we obtain
$$\left(\lim _{t \rightarrow 0}\left(\varphi_{t} * f\right)-f\right) * M_{\xi} {\phi^*}(x)=0$$
for almost every $x\in\mathbb R^n.$ Therefore, we can conclude that $\lim _{t \rightarrow 0} (\varphi_{t} * f)(x)=f(x)$ almost everywhere.
\end{proof}

\begin{remark}
In Proposition \ref{prop2}, we established a result analogous to Theorem \ref{th3}\,\eqref{th3-1}, although the boundedness of maximal function is not available in our setting. In this case, the result holds only in one direction, and no weight function has been considered. The main idea of the proof relies on the $L^p$-boundedness of the Hardy–Littlewood maximal operator applied to the convolution $f * |\phi^*|$. In general, for an arbitrary function belonging to a modulation space, the $L^p$-boundedness of $\mathcal{M}(f * |\phi^*|)$ requires that $|f| * |\phi^*|$ lies in some $L^p$ space. However, modulation spaces are not necessarily closed under the modulus operation, which makes it difficult to extend Proposition \ref{prop2} to all functions in modulation spaces.
\end{remark}

Moreover, by employing certain algebraic properties of modulation spaces, we can extend the validity of Proposition \ref{prop2} to arbitrary functions (not necessarily restrict non-negative) belonging to specific modulation spaces. We will discuss such spaces in the following result.

\begin{theorem}\label{th4}
Let $1 \leq p< \infty$ and $1 \leq q \leq 2.$ Then (\ref{eq16}) holds for all $f\in M^{p, q}\cap L^1_{loc}.$
\end{theorem}

\begin{proof}
For $f \in M^{p,q}$, we have
\begin{equation}\label{square}
(|f| * |\phi^*|)^2 \leq f^2 * |\phi^*|^2.
\end{equation}
If it can be shown that the square of any function in $M^{p,q}$ ($1 \leq p < \infty,\ 1 \leq q \leq 2$) lies in some other modulation space, then the right-hand side of \eqref{square} will belong to some $L^p$ space. This would imply
$|f| * |\phi^*| \in L^{2p}$,
from which the desired result follows by an argument analogous to Proposition \ref{prop2}.

The final step is to determine the modulation space containing the square $f^2$ of a given function $f$. By invoking the algebra property in Lemma \ref{lm1}~\eqref{ap}, we conclude that $f^2$ belongs to an appropriate modulation space, thereby completing the proof of Theorem \ref{th4}.

\end{proof}

\section{Heat equation with Laplacian}\label{sec4}
In this section, we recall the definition of the weight classes $D^h_{p,q}(\mathbb{R}^{2n})$ as introduced in the introduction. We then proceed to establish Theorem \ref{th1} (Laplacian part) by first proving a series of auxiliary lemmas that are essential to the overall argument.

\begin{remark}[proof strategy]\label{psrmt}
Let us briefly outline the key ideas used to characterize the weighted modulation spaces in which almost everywhere convergence holds (i.e., the strategy for proving Theorem \ref{th1}):
\begin{itemize}
\item[-]By proving Proposition \ref{prop1}, we can infer the appropriate weight classes in the context of modulation spaces.

\item[-]By Lemma \ref{lm6}, the weight class on $\mathbb{R}^{2n}$, as considered in Theorem \ref{th1}, can be suitably related to the previously studied weight class on $\mathbb{R}^n$ (see \eqref{Lp-weights}). This correspondence ensures that the weight class on $\mathbb{R}^{2n}$ satisfies the necessary structural conditions required to apply existing results from the theory of weighted Lebesgue spaces.

\item[-] Then, by Lemma \ref{lm5}, we deduce that if $f \in M_v^{p, q}$ is non-negative, there exists a test function $\phi \in \mathcal{S}$ such that $f * |\phi| \in L_v^{p}$ for some suitable $v$ on $\mathbb{R}^n$. This estimate allows us to apply known boundedness results within the framework of weighted Lebesgue spaces.
\item[-]Conversely, pointwise convergence ensures that \( f \) belongs to the required weighted space \( M_v^{p,q} \) (as shown in the proof of Proposition \ref{prop1}).

\item[-] However, in the case of the Hermite operator, additional effort is required to address this situation effectively; see Section \ref{sec5} for details.
\end{itemize}
\end{remark}

\begin{lemma}\label{lm6}
Let $v$ be a strictly positive weight in $\mathbb{R}^{2n}.$ Denote $v_\xi(x)=v(x,\xi)$ with $\xi$ fixed. If $v\in D^h_{p,q}(\mathbb{R}^{2n}),$ then $v_\xi\in D^h_p(\mathbb{R}^{n})$ for almost every \( \xi \in \mathbb{R}^n \).
\end{lemma}
\begin{proof}
Firstly, assume that $v \in D_{p, q}^{h}$. Then by definition there exist $t_{0}$ such that $h_{t_{0}} \in M_{v^{-1}}^{p, q}$. That is $V_{\phi} h_{t_{0}} \in L_{v^{-1}}^{p, q},$ where $V_{\phi} h_{t_{0}}$ is the short time Fourier transform of $h_{t_{0}}$ with respect to test function $\phi\in\mathcal{S}\backslash\{0\}.$ Without loss of generality we can choose $\phi=h_{t_{0}}.$
We can write $h_{t_{0}}(x)=D_{\sqrt{4 \pi {t_0}}}h(x),$ where the notation $D_{t} h(x)=t^{-n} h\left(t^{-1} x\right)$ with $h(x)=e^{-\pi |x|^2}$. Now consider $V_{h_{t_{0}}} h_{t_{0}}:$
\begin{align}\label{eq22}
V_{h_{t_{0}}} h_{t_{0}}(x, \xi) & =\int_{\mathbb R^{n}} D_{\sqrt{4 \pi t_0}} h(y) \,D_{\sqrt{4 \pi t_0}} h(y-x)\, e^{-2 \pi i y \cdot \xi} d y \nonumber\\
& =\int_{\mathbb R^{n}} D_{\sqrt{8 \pi t_0}} h(x) \,D_{\sqrt{2 \pi t_0}} h\left(y-\frac{x}{2}\right)\, e^{-2 \pi i y \cdot \xi} d y \nonumber\\
& =D_{\sqrt{8 \pi t_0}} h(x) \,\mathcal{F}\left({T_{x / 2} D_{\sqrt{2 \pi t_0}}}h\right)(\xi) \nonumber\\
& =D_{\sqrt{8 \pi t_0}} h(x)\, M_{x / 2}\left(\mathcal{F}(h)\right)(\sqrt{2 \pi t_0} \xi) \nonumber\\
& = e^{-2 \pi^{2} t_0 |\xi|^{2}}\,e^{-\pi i x \cdot \xi}\, h_{\frac{t_0}{2}}(x).
\end{align}
From equation \eqref{eq22}, we deduce that \( h_{\frac{t_0}{2}} \in L^p_{v_\xi} \) for almost every \( \xi \in \mathbb{R}^n \), since \( V_{h_{t_{0}}} h_{t_{0}} \in L_{v^{-1}}^{p, q} \). Hence we can conclude that \( v_\xi \in D^h_p \) for almost every \( \xi \in \mathbb{R}^n .\) This completes the proof.
\end{proof}

\begin{lemma}\label{lm5}
	Let $1 \leq p,q< \infty$ and  $v \in D^h_{p,q}(\mathbb{R}^{2n}).$ Suppose   $f \in M_v^{p, q}\cap L^1_{loc}$ be  
	non-negative. Then there exists some $\phi \in \mathcal S$
	such that $f * |\phi|\in L_v^{p}$ for some $v\in D^h_p(\mathbb{R}^n).$
    Moreover $\mathcal Mf * |\phi|\in L_u^{p}$ for some weight $u$ on $\mathbb{R}^{n}.$
\end{lemma}

\begin{proof}
Let $\rho \in L^{p'}(\mathbb{R}^n)$, where $p'$ is the Hölder conjugate of $p$. Denote $\eta(x)=e^{- |x|^2}.$ Consider
$$
\langle \eta(f * |\phi|),\,\rho \rangle = \int_{\mathbb{R}^n} \eta(x)\,(f * |\phi|)(x)\, \overline{\rho(x)} \, dx = \int_{\mathbb{R}^n} \int_{\mathbb{R}^n} \eta(x)\,f(y) \,|\phi(x - y)| \,\overline{\rho(x)} \, dy \, dx.
$$
Let's choose $\eta(x)=|\phi|(x)=e^{- |x|^2}.$ By Fubini's theorem and Hölder's inequality, we obtain
\begin{align}\label{dual}
|\langle \eta(f * |\phi|),\, \rho \rangle|
&\leq \int_{\mathbb{R}^n} f(y) \left( \int_{\mathbb{R}^n} e^{- |x|^2}e^{- |x-y|^2} |\rho(x)| \, dx \right) dy  \nonumber\\
&=  \int_{\mathbb{R}^n} f(y)\, e^{- \frac{|y|^2}{2}}\left( \int_{\mathbb{R}^n} e^{- |x-\frac{y}{2}|^2} |\rho(x)| \, dx \right) dy  \nonumber\\
&= \int_{\mathbb{R}^n} f(y)\,e^{- \frac{|y|^2}{2}} \left( \int_{\mathbb{R}^n} |T_{\frac{y}{2}} \phi(x)| \, |\rho(x)| \, dx \right) dy \nonumber\\
&\leq \int_{\mathbb{R}^n} f(y)\, e^{- \frac{|y|^2}{2}}\|T_{\frac{y}{2}} \phi\|_{L^p} \,\|\rho\|_{L^{p'}} \, dy \nonumber\\
&= \|T_{\frac{y}{2}} \phi\|_{L^p}\,\|\rho\|_{L^{p'}} \int_{\mathbb{R}^n} f(y) \, e^{- \frac{|y|^2}{2}} \, dy.
\end{align}
Let $f \in M_v^{p,q}$ with $v \in D^h_{p,q}(\mathbb{R}^{2n})$. We can choose $t_0 \leq \frac{1}{2}$ such that $\|h_{t_0}\|_{M_{v^{-1}}^{p', q'}} < \infty$, where $h_t$ denotes the heat kernel defined in \eqref{heatKernel}. Then, we have
$$
|\langle \eta(f * |\phi|), \rho \rangle| \lesssim \|T_{\frac{y}{2}} \phi\|_{L^p} \, \|\rho\|_{L^{p'}} \int_{\mathbb{R}^n} f(y) \, h_{t_0}(y) \, dy.
$$
Using the duality between $M_v^{p,q}$ and $M_{v^{-1}}^{p', q'}$, this yields
$$
|\langle \eta(f * |\phi|), \rho \rangle| \lesssim \|T_{\frac{y}{2}} \phi\|_{L^p} \, \|\rho\|_{L^{p'}} \, \|f\|_{M_v^{p,q}} \, \|h_{t_0}\|_{M_{v^{-1}}^{p', q'}}.
$$
Taking the supremum over all $\rho \in L^{p'}$ with $\|\rho\|_{L^{p'}} \leq 1$, we obtain
$$
\|\eta(f * |\phi|)\|_{L^p} \lesssim (2\pi)^{\frac{n}{2}} \, \|T_{\frac{y}{2}} \phi\|_{L^p} \, \|f\|_{M_v^{p,q}} \, \|h_{t_0}\|_{M_{v^{-1}}^{p', q'}}.
$$
Therefore, we conclude that $f * |\phi| \in L^p_{\eta^p}$, where $\eta^p \in D^h_p$.
\end{proof}

\begin{proposition}[Characterization of $D^h_{p,q}$]\label{prop1}
Suppose \( 1 \leq p ,q< \infty \). The weight \( v \in D^h_{p,q}(\mathbb{R}^{2n}) \)  if and only if there exists \( t_0 > 0 \) and a weight \( u \) on $\mathbb{R}^{2n}$ such that the operator \( f \mapsto h_{t_0} * f \) maps \( M_v^{p,q} \) into \( M_u^{p,q} \) with the norm inequality
\[ \| h_{t_0} * f \|_{M_u^{p,q}} \lesssim \| f \|_{M_v^{p,q}}. \]
\end{proposition}
We need the following lemma to prove Proposition \ref{prop1}.

\begin{lemma}\label{lm4} Suppose \( 1 \leq p, q < \infty \). Define 
\[ g_{t}(x) = \left\|h_{t}(x - \cdot)\right\|_{M_{v^{-1}}^{p', q'}} \quad\text{for }t>0.\]
If \( v \in D^h_{p,q}(\mathbb{R}^{2n}) \), then there exist \( t_0>0 \) and a weight \( u \) on $\mathbb{R}^{2n}$ such that $\left\|g_{t_0}\right\|_{M_u^{p, q}} < \infty.$
\end{lemma}
\begin{proof}
Consider the function $g_t(x)$, which can be expressed as follows
\[
g_t(x) = \left\| h_t(x - \cdot) \right\|_{M_{v^{-1}}^{p', q'}} \lesssim \left\| V_{h_t} h_t(x - \cdot, \cdot) \right\|_{L_{v^{-1}}^{p', q'}}.
\]
Using equation \eqref{eq22}, we can write \( g_{t_0}(x) \) as
\begin{equation}\label{eq38}
g_{t_0}(x) = \left( \int_{\mathbb{R}^n} \left( \int_{\mathbb{R}^n} \left( e^{-2 \pi^2 {t_0} |\xi|^2} \, h_{\frac{t_0}{2}}(x - y) \right)^p v(y, \xi) \, dy \right)^{\frac{q}{p}} d\xi \right)^{\frac{1}{q}}
\end{equation}
for some $t_0>0$ (from Lemma \ref{lm6}). Next, by using the conditions
\[
\begin{cases}
|x - y| \leq |x| & \Rightarrow |y| \leq |x - y| + |x| \leq 2|x| \\
|x - y| > |x| & \Rightarrow |y| \leq |x - y| + |x| \leq 2|x - y|,
\end{cases}
\]
we can deduce that \( h_{t_0}(x - y) \lesssim h_{\frac{t_0}{4}}(y) \).
Substituting this estimate into equation \eqref{eq38}, we obtain
\[
g_{t_0}(x) \lesssim \left\| V_{h_{\frac{t_0}{8}}} h_{\frac{t_0}{8}} \right\|_{L_{v^{-1}}^{p', q'}}.
\]
Hence, it follows that \( g_{t_0}(x) < \infty \) a.e. for \( x \), since we are given that 
\[ \left\| h_{\frac{t_0}{8}} \right\|_{M_{v^{-1}}^{p', q'}} < \infty. \]
Now, we will choose a weight \( u \in L_{\text{loc}}^1 \) such that \( g_{t_0} \in M_u^{p, q} \), which is equivalent to showing \( V_\phi g_{t_0} \in L_u^{p, q} \) for some $\phi\in\mathcal{S}.$ By choosing \( u \in L_{\text{loc}}^1\) as
\[ 
u(x, \xi) = 
\begin{cases}
	1 & \text{if } |V_\phi g_{t_0}(x, \xi)| \leq u'(x, \xi) \\
	\frac{u'(x, \xi)}{|V_\phi g_{t_0}(x, \xi)|} & \text{if } |V_\phi g_{t_0}(x, \xi)| > u'(x, \xi),
\end{cases}
\]
where $u'\in L^{p,q},$ we get that \( g_{t_0} \in L_u^{p, q} \). This means we are choosing \( u \) in such a way that \( |V_\phi g_{t_0}| \cdot u\) is dominated by a function in \( L^{p, q} \). 
\end{proof}

\begin{proof}[\bf Proof of Proposition \ref{prop1}]
	By using the Moyal identity (Lemma \ref{lm1} \eqref{eq37}) and H{\"o}lder's inequality, for any \( t > 0 \), we have
	\begin{eqnarray}
		\left\|h_t * f\right\|_{M_u^{p, q}}\nonumber
		& = & \left\|\int h_t(\cdot - y) f(y) \, dy\right\|_{M_u^{p, q}} \\\nonumber
		& \lesssim &  \left\|\int V_\phi h_t(\cdot - y) V_\phi f(y) \, dy\right\|_{M_u^{p, q}} \\\nonumber
        & \leq &   \left\| \left\|f\right\|_{M_v^{p, q}}g_{t}(.)\right\|_{M_u^{p, q}}=\left\|f\right\|_{M_v^{p, q}} \left\|g_{t}\right\|_{M_u^{p, q}}, 
	\end{eqnarray}
where $g_t$ is as defined in Lemma~\ref{lm4}.
 Hence, we have \( \left\|h_{t_0} * f\right\|_{M_u^{p, q}} \lesssim \left\| f \right\|_{M_v^{p, q}} \) for some \( t_0 >0\), since Lemma \ref{lm4} guarantees that there exist \( t_0 \) and a weight \( u \) such that $\left\|g_{t_0}\right\|_{M_u^{p, q}} < \infty.$
 
\smallskip
	
	Conversely, let 
    \[ \left\|h_{t_0} * f\right\|_{M_u^{p, q}} \lesssim \left\| f \right\|_{M_v^{p, q}} \] 
    for all \( f \in M_v^{p, q} \). Thus, \( |h_{t_0} * f(x)| < \infty \) almost everywhere for all \( f \in M_v^{p, q} \). Fix \( x_0 \) such that \( h_{t_0} * f(x_0) < \infty \). Then we will show that \( h_{\frac{t_0}{4}} * f(x) < \infty \) for all \( x \). 
Let assume \(x \neq x_{0}\). Note that 
\begin{equation}\label{de1}
\begin{cases} |x-y| \leq |x-x_{0}| \implies\left|y-x_{0}\right| \leq 2\left|x-x_{0}\right|\\
 |x-y| \geq |x-x_{0}| \implies  \left|x_{0}-y\right| \leq 2|x-y|.
\end{cases}
\end{equation}
By \eqref{de1},  we obtain
\begin{eqnarray*}h_{\frac{t_{0}}{4}}(x-y) \lesssim
\begin{cases} \frac{h_{t_{0}}\left(x_{0}-y\right)}{h_{t_{0}}\left(2\left(x-x_{0}\right)\right)} \lesssim h_{t_{0}}\left(x_{0}-y\right) & if \  |x-y| \leq |x-x_{0}| \\
 h_{\frac{t_{0}}{4}}\left(\frac{x_{0}-y}{2}\right) \lesssim h_{t_{0}}\left(x_{0}-y\right) & if \ |x-y| \geq |x-x_{0}|.
 \end{cases}
\end{eqnarray*}
Thus,  we have 
\begin{eqnarray*}
h_{\frac{t_0}{4}} \ast f(x)
& \lesssim & \left(\int_{|x-y|<|x-x_{0}|} + \int_{|x-y| \geq |x-x_{0}|}\right) h_{t_{0}}\left(x_{0}-y\right) f(y) d y\nonumber\\
& \lesssim & h_{t_0}\ast f (x_0)
 <  \infty  
\end{eqnarray*}
 for all $x \in \mathbb{R}^{n} \backslash \{x_{0}\}$. Note that, later inequality in \eqref{de1} also hold for $x=x_0.$ 
That is
\[0 \leq \int_{\mathbb{R}^{n}} h_{\frac{t_{0}}{4}}\left(x_{0}-y\right) f(y) d y \lesssim \int_{\mathbb{R}^{n}} h_{t_{0}}\left(x_{0}-y\right) f(y) d y < \infty.\]
Thus,   we have
\[\int_{\mathbb{R}^{n}} h_{\frac{t_{0}}{4}}(x-y) f(y) d y< \infty \quad for\ all\  \ x\in \mathbb R^n.\] 
In particular, \( h_{\frac{t_0}{4}} * f(0) < \infty \), i.e.,
	\[ \int h_{\frac{t_0}{4}}(y) f(y) \, dy < \infty. \]
By duality, \( h_{\frac{t_0}{4}} \in M_{v^{-1}}^{p', q'} \), we can conclude that \( v \in D^h_{p, q} \).
\end{proof}

Now we are ready to prove our main result.
\begin{proof}[\bf Proof of Theorem \ref{th1}~(Laplacian case)]
Assume that \( v \in D^h_{p,q} \) and consider the following expression
\begin{align}\label{eq11}
		\lim_{t \rightarrow 0} \left( h_{t} * (f * M_{\xi}  {\phi^*})\right)(x)
		&= \lim_{t \rightarrow 0} \left(( h_{t} * f) * M_{\xi}  {\phi^*}\right)(x) \nonumber\\
		&= \lim_{t \rightarrow 0} \int_{\mathbb{R}^n}  (h_{t} * f)(y) M_{\xi}  {\phi^*}(x - y) \, dy.
	\end{align}
We aim to use the dominated convergence theorem (DCT) on the right-hand side of equation (\ref{eq11}) and interchange the limit and integration. To do this, we define a sequence of functions by fixing $x$ as follows
\[ F_t(y) :=   (h_{t} * f) (y) M_{\xi}  {\phi^*}(x - y). \]

Now, applying Theorem \ref{th5}, we obtain the inequality
\begin{equation}\label{eq13}
	|F_t(y)| \leq A\, |\mathcal{M}f(y) M_{\xi}  {\phi^*}(x - y)| \leq A\, \mathcal{M}f(y) | {\phi^*}|(x - y)
\end{equation}
for all \( t > 0 \). In order to utilize the DCT, we need the right-hand side of equation (\ref{eq13}) to be an integrable function. Employing Lemma \ref{lm5}, we observe that
\[ \int_{\mathbb{R}^n} \mathcal{M}f(y) | {\phi^*}|(x - y) \, dy = \mathcal{M}f * | {\phi^*}|(x) < \infty \]
for almost every \( x \). Consequently, we can apply the DCT in equation (\ref{eq11}) to obtain
\begin{align}\label{eq15}
		\lim_{t \rightarrow 0} \left( h_{t} * (f * M_{\xi}  {\phi^*})\right)(x)
		&= \int_{\mathbb{R}^n} \lim_{t \rightarrow 0} \left( h_{t} * f\right)(y) M_{\xi}  {\phi^*}(x - y) \, dy \nonumber\\
		&= \left(\lim_{t \rightarrow 0} ( h_{t} * f)\right) * M_{\xi}  {\phi^*}(x).
	\end{align}
However, by Lemma \ref{lm5}, which states that $v \in D^h_{p,q}$ implies $v_\xi \in D^h_p$ for almost every $\xi \in \mathbb{R}^n$, it follows that $f \in M_v^{p, q}$ implies $f * M_{\xi} {\phi^*} \in L_{v_\xi}^{p}$ for almost every $\xi \in \mathbb{R}^n$. Consequently, by applying Theorem \ref{viviT}, which states that \( v \in D^h_p(\mathbb{R}^{n}) \) if and only if \(\lim_{t \to 0} (h_t * f)(x) = f(x) \quad \text{a.e. for all } f \in L^p_v(\mathbb{R}^n),\)
we obtain
\begin{equation}\label{eq21}
	\lim_{t \rightarrow 0} \left( h_{t} * \left(f * M_{\xi}  {\phi^*}\right)\right)(x) = f * M_{\xi}  {\phi^*}(x).
\end{equation}

Now, comparing equations (\ref{eq15}) and (\ref{eq21}), we derive
\[ \left(\lim_{t \rightarrow 0} \left( h_{t} * f\right) - f\right) * M_{\xi}  {\phi^*}(x) = 0 \]
for almost every \( x \in \mathbb{R}^n \). Therefore, we can conclude that \( \lim_{t \rightarrow 0} ( h_{t} * f) = f \) almost everywhere.

\smallskip

Conversely, let us assume \( \lim_{t \rightarrow 0} h_{t} * f(x) = f(x) \) almost everywhere for all non-negative \( f \in M_v^{p, q} \). By Proposition \ref{prop1}, there exists \( t_0>0 \) such that \(  h_{t_0} * f(x) < \infty \) almost everywhere for all non-negative \( f \in M_v^{p, q} \). Fix \( x_0 \) such that \(  h_{t_0} * f(x_0) < \infty \). Following the same approach as in the converse part of Proposition \ref{prop1}, we can similarly obtain that
$h_{\frac{t_0}{4}} * f(x) < \infty$
for all \( x \in \mathbb{R}^n \), i.e.,
\[ \int_{\mathbb{R}^n} h_{\frac{t_0}{4}}(y) f(y) \, dy < \infty. \]
By duality, \( h_{\frac{t_0}{4}} \in M_{v^{-1}}^{p', q'} \), which implies that \( v \in D^h_{p,q} \).

\end{proof}

\section{Heat equation with Hermite operator}\label{sec5}
In this section, we consider the heat semigroup $e^{-t H} f$ and derive their simplified expressions. We then proceed to prove our main result, Theorem~\ref{th1} (Hermite part).

\smallskip

Consider the heat equation \eqref{eq9} associated with the harmonic oscillator \( H \), whose heat semigroup is given by \eqref{H-heatkernel}.
The spectral decomposition of the Hermite operator is expressed in \eqref{sd}, where the eigenfunctions are the normalized Hermite functions
\[\Phi_\alpha(x) = \prod_{j=1}^n \tilde{h}_{\alpha_j}(x_j)\]
and each \( \tilde{h}_k \) denotes the one–dimensional normalized Hermite function, given by
\[\tilde{h}_k(x)
= \left( \sqrt{\pi}, 2^k k! \right)^{-1/2} (-1)^k e^{\frac{x^2}{2}} \frac{d^k}{dx^k}\big( e^{-x^2} \big).\]
These functions form an orthonormal basis for $L^2$ and satisfy the eigenvalue relation $$H \Phi_\alpha = (2|\alpha| + n)\Phi_\alpha$$ where $|\alpha| = \alpha_1 + \dots + \alpha_n$. A fundamental identity involving Hermite functions is Mehler's formula (see \cite[Lemma 1.1.1]{MR1215939}), which for $|w| < 1$ takes the form
\begin{equation}\label{eq26}
\sum_{k=0}^{\infty} \frac{\tilde{h}_k(x) \tilde{h}_k(y)}{2^k k!} w^k=\left(1-w^2\right)^{-\frac{1}{2}} e^{-\frac{1}{2} \frac{1+w^2}{1-w^2}\left(x^2+y^2\right)+\frac{2 w}{1-w^2} x y}.
\end{equation}
Using Mehler's formula \eqref{eq26}, we can rewrite the heat semigroup \eqref{H-heatkernel} as following
\begin{equation}\label{eq28}
e^{-t H} f(x) = \int_{\mathbb{R}^n} h_t^H(x, y) f(y) \, dy = \int_{\mathbb{R}^n} \frac{e^{-\left[\frac{1}{2}|x-y|^2 \coth 2t + x \cdot y \tanh t\right]}}{(2 \pi \sinh 2t)^{\frac{n}{2}}} f(y) \, dy
\end{equation}
as shown in \cite[Eq. (4.1.2)]{MR1215939}. By applying Stefano Meda's change of parameters
\[t = \frac{1}{2} \log \frac{1+s}{1-s} \quad \text{for } t \in (0, \infty) \text{ and } s \in (0,1),\]
equivalently $s=\tanh{t},$ we obtain
\begin{equation}\label{eq33}
h_{t}^H(x, y) = \left( \frac{1-s^2}{4 \pi s} \right)^{\frac{n}{2}} e^{-\frac{1}{4}\left[s|x+y|^2 + \frac{1}{s}|x-y|^2\right]}.
\end{equation}
In the limit as \( s \to 0^{+} \), it follows that \( t \to 0^{+} \) as well. From equation \eqref{eq33}, we observe that the following inequality holds
\begin{equation}\label{eq34}
h_{t}^H(x, y) \leq (1 - s^2) h_s(x - y),
\end{equation}
where \( h_s \) denotes the classical heat kernel \eqref{heatKernel}. Furthermore, using the conditions
\[
\begin{cases} 
2|x| < |y| & \Rightarrow |x + y| \leq 3|x - y| \\
2|x| \geq |y| & \Rightarrow |x + y| \leq 3|x|,
\end{cases}
\]
we can deduce the following estimates
\[
\begin{aligned}
& \left\{
\begin{array}{l}
e^{-\frac{1}{4}\left[s |x + y|^2 + \frac{1}{s} |x - y|^2 \right]} \geq e^{-\left( \frac{9 s^2 + 1}{s} \right) \frac{|x - y|^2}{4}} \\
e^{-\frac{1}{4}\left[s |x + y|^2 + \frac{1}{s} |x - y|^2 \right]} \geq e^{-\frac{9s}{4} |x|^2} e^{-\left( \frac{9 s^2 + 1}{s} \right) \frac{|x - y|^2}{4}}.
\end{array}
\right.
\end{aligned}
\]
Thus, for any \( (x, y) \) and \( 0 < s < 1 \), with the relation \( s = \tanh t \), the following inequality holds
\begin{equation}\label{eq36}
h_{t}^H(x, y) \geq e^{-\frac{9s}{4} |x|^2} \left( \frac{1 - s^2}{1 + 9 s^2} \right)^{\frac{n}{2}} h_{\frac{s}{1 + 9 s^2}}(x - y).
\end{equation}
In the following, we present the proof of the Hermite heat part of Theorem~\ref{th1}.

\begin{proof}[{\bf Proof of Theorem~\ref{th1}~(Hermite case)}]
Let \( f \) be a non-negative function in \( M^{p, q}_v \).  
By Lemma~\ref{lm5}, there exists a test function \( \phi \in \mathcal{S} \) such that \( f * |\phi| \in L_v^p \) for some weight \( v \in D^h_p \).  
To prove that \( \lim_{t \to 0} e^{-t H} f(x) = f(x) \) almost everywhere, it suffices to show that
\begin{equation}\label{eq53}
\left( \lim_{t \to 0} e^{-t H} f \right) * |\phi|(x) = f * |\phi|(x) \quad \text{almost everywhere}.
\end{equation}
Now, consider the left-hand side. By applying a method similar to that used in \eqref{eq11}-\eqref{eq13} and invoking the dominated convergence theorem, we can express
\[\left(\lim_{t \rightarrow 0} e^{-t H} f\right) * |\phi|(x) = \lim_{t \rightarrow 0} \left(e^{-t H} f * |\phi|\right)(x)\quad\mbox{almost everywhere}.\]
Now, we have
\begin{equation*}\label{}
	\begin{aligned}
		\lim_{t \rightarrow 0} (e^{-t H} f * |\phi|)(x)
		&= \lim_{t \rightarrow 0}\int_{\mathbb{R}^n}  e^{-t H} f(y) |\phi|(x - y) \, dy\\
		&= \lim_{t \rightarrow 0}\int_{\mathbb{R}^n} \int_{\mathbb{R}^n} h_{t}^H(y, y') f(y') \, |\phi|(x - y) \, dy'dy.
	\end{aligned}
\end{equation*}
Using upper bound \eqref{eq34} applying the associativity of convolution, we obtain
\begin{equation*}\label{}
	\begin{aligned}
 \lim_{t \rightarrow 0} (e^{-t H} f * |\phi|)(x)
		\leq&\lim_{s \rightarrow 0}\int_{\mathbb{R}^n} \int_{\mathbb{R}^n} (1 - s^2) h_s(y - y') f(y') \,  |\phi|(x - y) \, dy'dy \\     
        =&\lim_{s \rightarrow 0}(1 - s^2)\left((h_s\ast f)\ast |\phi|\right)(x)\\
        =&\lim_{s \rightarrow 0}\left(h_s\ast (f\ast |\phi|)\right)(x).
	\end{aligned}
\end{equation*}
However, by Lemma \ref{lm6}, we know that \( f * |\phi| \in L_{v}^{p} \) for some weight \( v \in D^h_p \). Consequently, using Theorem \ref{viviT}, we arrive
\begin{equation}\label{eq51}
\left(\lim_{t \rightarrow 0} e^{-t H} f\right) * |\phi|(x)\leq f * |\phi|(x)\text{ a.e. }x\in\mathbb{R}^n.
\end{equation}
Proceeding in a similar manner and utilizing the lower bound from equation \eqref{eq36}, we obtain the inequality
\begin{equation}\label{eq52}
\left( \lim_{t \to 0} e^{-t H} f \right) * |\phi|(x) \geq f * |\phi|(x) \quad \text{a.e. } x \in \mathbb{R}^n.
\end{equation}
Thus, combining \eqref{eq51} and \eqref{eq52}, we obtain the conclusion \eqref{eq53}, and hence establish the pointwise convergence almost everywhere.

Conversely, let us assume \( \lim_{t \rightarrow 0} e^{-t H} f(x) = f(x) \) almost everywhere for all non-negative \( f \in M_v^{p, q} \). Then, by \eqref{eq34} there exists \( t_0>0 \) such that 
\[e^{-t H} f(x)\lesssim  h_{t_0} * f(x) < \infty \] 
almost everywhere for all non-negative \( f \in M_v^{p, q} \). Following the same approach as in the converse part of Proposition  \ref{prop1}, we can similarly deduce that
\( h_{\frac{t_0}{4}} \in M_{v^{-1}}^{p', q'} \), which implies  \( v \in D^h_{p,q} \).
\end{proof}

\section{Concluding Remarks}\label{crf}

 Our main Theorems \ref{th1} and \ref{th7} naturally opens the door for the further investigation. In view of  ongoing research interest  mentioned in the introduction, we raise  some interesting open questions in the following points.   
 
\begin{enumerate}
\item Can we relax the  the non-negativity assumption on the initial data in Theorem \ref{th1}? The present proofs rely on positivity arguments and maximal function techniques. However, it would be of independent interest to determine whether the equivalence remains valid for general (possibly sign-changing) functions in weighted modulation spaces. See Remark \ref{rsi}.

\item It would be worthwhile to extend the analysis to other semigroups  such as Poisson-type semigroups or semigroups generated by more general differential operators. See for e.g. the discussion following Theorem \ref{viviT}. In such settings, additional analytical difficulties may arise due to weaker kernel regularity or the absence of explicit representations.

\item  Up to now our knowledge on maximal operator on modulation spaces is quite limited (see Theorem \ref{th7}).  The systematic investigation of the boundedness properties of maximal operators in modulation spaces remains an important open direction. 

\item  The combination of semigroup kernel estimates with time-frequency techniques provides a flexible method for deriving pointwise convergence theorems in non-classical function spaces.  This shed  new light on how large (rough)  initial data one can take,  see Remark \ref{example}.   We expect that the ideas developed here may serve as a foundation for further results on almost everywhere convergence and related fine properties of other evolution equations  in the modulation space setting.
\end{enumerate}

\subsection*{Acknowledgement} The second author gratefully acknowledges the support provided by IISER Pune, Government of India. We sincerely thank the referees for their careful reading of our manuscript and for their valuable comments and suggestions, which have significantly improved the presentation of the paper.

\bigskip

\bibliographystyle{plain}
\bibliography{bibliography.bib}

\end{document}